\documentclass[a4paper, 12pt]{amsart}
\usepackage[svgnames]{xcolor}
\usepackage{float}
\usepackage[pdftex]{graphicx}
\usepackage{tikz}
\usetikzlibrary{intersections, calc, arrows}
\usetikzlibrary{positioning}
\usetikzlibrary{chains}
\usepackage{epstopdf}
\allowdisplaybreaks[4]

\usepackage{amsmath,amssymb,amsthm}
\usepackage{amscd}
\usepackage{amsmath}
\usepackage{amssymb}
\usepackage{amsfonts}
\usepackage{bm}
\usepackage{cases}
\usepackage{color}
\usepackage{comment}
\usepackage{mathrsfs}
\usepackage{amsthm}
\usepackage{enumerate}
\usepackage[all]{xy} 
\xyoption{pdf}
\usepackage{setspace}
\usepackage{latexsym}
\usepackage{wrapfig}
\usepackage{url}
\usepackage{braket}
\usepackage{array,enumerate}

\DeclareMathAlphabet{\mathpzc}{OT1}{pzc}{m}{it}
\theoremstyle{definition}
\newtheorem{defi}{Definition}[section]
\newtheorem{exam}[defi]{Example}

\theoremstyle{theorem}
\newtheorem{theo}[defi]{Theorem}
\newtheorem*{theo*}{Theorem}
\newtheorem{prop}[defi]{Proposition}
\newtheorem{coro}[defi]{Corollary}
\newtheorem{lemm}[defi]{Lemma}

\theoremstyle{remark}
\newtheorem{rema}[defi]{Remark}

\newcommand{\Natural}{\mathbb{N}}
\newcommand{\Nonnegative}{\mathbb{N}_0}
\newcommand{\Integer}{\mathbb{Z}}

\newcommand{\Complex}{\mathbb{C}}
\newcommand{\abs}[1]{\left\lvert #1 \right\rvert}
\newcommand{\norm}[1]{\left\lVert #1 \right\rVert}

\newcommand{\setcond}[2]{\left\{ #1 \,;\, #2 \right\}}

\newcommand{\diam}{\mathrm{diam}}

\renewcommand{\set}[1]{\left\{ #1 \right\}}

\makeatletter

\@addtoreset{equation}{section}

\title[Adjacency and transition matrices]{The adjacency matrices and the transition matrices related to random walks on graphs}
\author[T. Ikkai]{Tomohiro Ikkai}
\author[H. Ohno]{Hiromichi Ohno}
\author[Y. Sawada]{Yusuke Sawada}
\address{T. Ikkai: Graduate school of mathematics, Nagoya University, Furocho, Chikusaku, Nagoya, 464-8602, Japan}
\email{m13006z@math.nagoya-u.ac.jp}
\address{H. Ohno: Department of Mathematics, Faculty of Engineering, Shinshu University, 4-17-1 Wakasato, Nagano 380-8553, Japan}
\email{h\_ohno@shinshu-u.ac.jp}
\address{Y. Sawada: Graduate school of mathematics, Nagoya University, Furocho, Chikusaku, Nagoya, 464-8602, Japan}
\email{m14017c@math.nagoya-u.ac.jp}
\begin{document}
\maketitle
\begin{abstract}
A pointed graph $(\Gamma,v_0)$ induces a family of transition matrices in Wildberger's construction of a hermitian hypergroup via a random walk on $\Gamma$ starting from $v_0$. We will give a necessary condition for producing a hermitian hypergroup as we assume a weaker condition than the distance-regularity for $(\Gamma,v_0)$. The condition obtained in this paper connects the transition matrices and the adjacency matrices associated with $\Gamma$.
\end{abstract}


\section{Introduction} \label{secIntro}

The concept of (locally compact) hypergroups was introduced in 1970s. 
It is said that hypergroups originated in statistics, 
and they have an aspect as a probabilistic extension of locally compact groups. 
We can find an account on harmonic analysis and representation theory 
of hypergroups, as analogies of those of topological groups, in Bloom and Heyer's monograph \cite{bloo-heye95}. 

Hypergroups are sometimes utilized to describe probabilistic phenomena. 
One of the interesting examples of utilization is Wildberger's study \cite{wild94, wild95} 
on the connection between finite (discrete) hypergroups and random walks on graphs. 
What he did is producing certain finite hypergroups connected with random walks 
on the platonic solid graphs. 

Subsequently, the first and third authors formulated 
Wildberger's construction of hypergroups in \cite{ikka-sawa19}.  
They revealed that his method can be applied to any distance-regular graphs 
including some infinite graphs. 
Moreover, they discovered several examples that produce discrete hypergroups 
in the same way as his. 

Wildberger's method requires that a graph $\Gamma = (V, E)$, 
where $V$ denotes the vertex set and $E$ the edge set of $\Gamma$, is at least simple, connected and self-centered. (A weaker condition given in \cite{endo-mimu-sawa20}). 
Meeting these requirement, 
a random walk on a graph induces a convolution $\circ$ and an involution $*$ 
to a certain vector space over $\Complex$. 
If this convolution is both associative and commutative, 
that random walk on $\Gamma$ is said to produce a discrete hypergroup. 
We need to note that not all random walks produce discrete hypergroups. 
Here a problem arises: On what kind of graphs random walks produce 
discrete hypergroups in his method? 
We shall say it to be ``hypergroup productive'' 
a pair $(\Gamma, v_0)$ of a graph $\Gamma$ and a vertex $v_0$
if a random walk on $\Gamma$ starting from $v_0$ produces a discrete hypergroup 
(cf.\ Definition \ref{defHGP}). 

The first attempt to solve this problem is Endo-Mimura-Sawada's study 
\cite{endo-mimu-sawa20}. 
Investigating the distance distributions of random walks on graphs, 
they conjectured that a graph-vertex pair $(\Gamma, v_0)$ would be 
hypergroup productive when $(\Gamma, v_0)$ satisfied the good symmetry conditions, 
somewhat weaker than the distance-regularity. 
Those conditions will be introduced as (S1) and (S2) later. 

The investigations in this paper will provide a matrix theoretical approach 
to solving the problem. 
The authors discovered that the associativity and the commutativity of 
the convolution $\circ$ derived from a random walk can be written in terms of 
the ``$k$-adjacency matrices'' $A^{(k)}$'s, defined as 
\begin{equation*}
A^{(k)} = \left( A^{(k)}_{x, y} \right)_{x, y \in V} = 
\begin{cases}
1 & (d(x, y) = k), \\
0 & (d(x, y) \neq k), 
\end{cases}
\end{equation*}
of the given graph.  




The main results in this paper imply the following theorem. 
\begin{theo*}
Let $\Gamma = (V, E)$ be a graph and $v_0 \in V$. 
Suppose that the graph-vertex pair $(\Gamma, v_0)$ satisfies 
the good symmetry conditions (S1) and (S2). 

If the $k$-adjacency matrices $A^{(k)}$'s of $\Gamma$ mutually commute, that is, 
$A^{(k)} A^{(l)} = A^{(l)} A^{(k)}$ holds for every $k$, $l$ 
with $0 \leq k, l < \diam(\Gamma)+1$, 
then the pair $(\Gamma, v_0)$ is hypergroup productive. 
(The symbol $\diam(\Gamma)$ denotes the diameter of $\Gamma$.)
\end{theo*}
The precise form of this theorem will be stated as Theorem \ref{assoA_k}. 

It has been necessary troublesome calculations of the convolutions so far 
in order to determine whether the given graph-vertex pair $(\Gamma, v_0)$ 
is hypergroup productive. 
This theorem enables us to check the hypergroup productivity of $(\Gamma, v_0)$ 
without any calculations of the convolutions.

This study is motivated by the theory of association schemes, 
the main topic in algebraic combinatorics. 

It is well known that every association scheme provides an algebra of matrices, 
called the Bose-Mesner algebra. 
Wildberger \cite{wild95} mentioned that a Bose-Mesner algebra always yields 
a discrete hypergroup, 
so one has a discrete hypergroup derived from an association scheme. 
On the other hand, every association scheme corresponds to a distance-regular graph. 
It was proved in \cite{ikka-sawa19} that a graph-vertex pair $(\Gamma, v_0)$ is hypergroup productive whenever $\Gamma$ is a distance-regular graph, 
so one has another discrete hypergroup derived from an association scheme. 
In fact, these two hypergroups are isomorphic, which was proved in \cite{ikka-sawa19}. 

This fact tells us that the discrete hypergroup derived from a distance-regular graph 
$\Gamma$ contains the same combinatorial infomation of $\Gamma$ 
as the corresponding association scheme. 
One can infer that the discrete hypergroup derived from a hypergroup productive graph-vertex pair $(\Gamma, v_0)$ contains some combinatorial infomation of 
$\Gamma$. 
Hence it is significant to investigate such discrete hypergroups. 
The authors' study in this paper should be regarded as the foundation of studies 
on the connections between hypergroups and random walks on graphs. 

As the conclusion of this section, the contents of this paper will be shown here. 

Section \ref{secMethod} is devoted to explaining the fundamental theory 
based on the preceding studies \cite{endo-mimu-sawa20, ikka-sawa19, wild94, wild95}. 
That section serves also the definitions and the notations from graph theory 
and hypergroup theory for the following sections. 

Section \ref{secMain} furnishes one with the main results, 
which yield the above theorem. 
Two kinds of matrices will play a key role. 
One is the $k$-adjacency matrices $A^{(k)}$, which was introduced above. 
The other is the $k$-transition matrices $P_k$, which describes the distance distribution of the random walks in question.  
Theorems \ref{PD=DA} and \ref{assoA_k} are so important that 
they are the keys of the proof of the above theorem. 

As an application of the main results, we discuss the simplest case in Section 
\ref{secApplication}, when a given graph is of diameter two. 
In this case, we can determinate whether a given graph-vertex pair with the conditions (S1) and (S2) produces a discrete hypergroup by the above theorem. 
However, we need some combinatorial calculations to obtain the structure 
of the hypergroup derived from that graph-vertex pair. 


\section{Wildberger's construction} \label{secMethod}

Let $\Natural$, $\Nonnegative$, $\Integer$ and $\Complex$ denote 
the set of all positive integers, $\Natural \cup \set{0}$, the ring of all integers 
and the field of the complex numbers, respectively. 
In addition, we use the following notations. 
\begin{itemize}
\item 
For a set $S$, let $\abs{S}$ denote the cardinality of $S$. 
We do not distinguish the cardinalities of infinite sets, 
and therefore we write simply $\abs{S} = \infty$ if $S$ is an infinite set. 

\item
For a set $S$, let $\Complex S$ denote the free vector space generated by $S$ 
over $\Complex$, that is, 
\begin{equation*}
\Complex S = \setcond{\sum_{i=1}^{n} a_i x_i}{n \in \Natural, 
a_1, \dots, a_n \in \Complex, x_1, \dots, x_n \in S}. 
\end{equation*}

\item
Let $\delta_{i, j}$ denote the Kronecker delta 
\begin{equation*}
\delta_{i, j} = 
\begin{cases}
1 & (i = j), \\
0 & (i \neq j). 
\end{cases}
\end{equation*}

\item
For a set $S$, let $\Complex^S = \setcond{(a_i)_{i \in S}}{a_i \in \Complex}$ denote 
the complex vector space consisting of all sequences indexed by $S$. 

\item 
Let $\bm{1}$ denote the identity matrix. 
Note that the order of $\bm{1}$ depends on the context. 
\end{itemize}

In this section, we will prepare notations related with graphs 
and recall Wildberger's method of construction of a discrete hypergroup 
from a given graph. 

\subsection{Definitions and notations from graph theory}

We refer the reader to \cite{bigg93} for the basic notions of graphs. 
In this paper, we assume that any graph is simple, connected, locally finite, 
self-centered, and has at most countably many vertices. 
We use the notation $d(v, w)$ to express the distance on a graph 
between two vertices $v$ and $w$. 

Let $\Gamma$ be a graph with the vertex set $V$. 
For each $k \in \Nonnegative$ and $v \in V$, we define the subset $S_k(v)$ of $V$ as 
\begin{equation*}
S_k(v) = \setcond{w \in V}{d(v, w) = k}. 
\end{equation*}
Fixing a vertex $v_0 \in V$ as a \textit{base point}, 
we call the pair $(\Gamma, v_0)$ a \textit{pointed graph}. 
We define the set $I(\Gamma, v_0)$ of non-negative integers as 
\begin{equation*}
I(\Gamma, v_0) = \set{d(v_0, v)}_{v \in V} = 
\begin{cases}
\set{0, 1, \dots, \diam(\Gamma)} & (\abs{V} < \infty), \\
\Nonnegative & (\abs{V} = \infty), 
\end{cases} 
\end{equation*}
where $\diam(\Gamma)$ denotes the diameter of $\Gamma$, 
defined by
\[
\diam(\Gamma) = \sup \setcond{d(v, w)}{v, w \in V}.
\] 
Under the assumption that $\Gamma$ is self-centered, 
it is obvious that $S_k(v) = \varnothing$ when $k \notin I(\Gamma, v_0)$. 


Here we give the definition of ``$k$-adjacency matrices,'' an extended concept. 

\begin{defi}
Let $\Gamma$ be a graph with the vertex set $V$ and $k \in I(\Gamma,v_0)$. 
We call the symmetric matrix 
\begin{equation*}
A^{(k)} = (\delta_{k, d(x, y)})_{x, y \in V} 
\end{equation*}
the \textit{$k$-adjacency matrix} associated with $\Gamma$. 
\end{defi}

Note that the $0$-adjacency matrix $A^{(0)}$ is the identity matrix. 
The $1$-adjacency matrix $A^{(1)}$ coincides with the ordinary adjacency matrix. 
Note that the sum $\sum_{k \in I(\Gamma, v_0)} A^{(k)}$ of the $k$-adjacency matrices 
is equal to the matrix $J$ all of whose entries are $1$. 


In this paper, unless otherwise specifically noted, 
the symbols $V$ and $A^{(k)}$ denote the vertex set of a graph $\Gamma$ 
and the $k$-adjacency matrix associated with $\Gamma$, respectively.


\subsection{Discrete hypergroups} 

The concept of hypergroups is the probability theoretic extension 
of locally compact groups, 
introduced in \cite{dunk73}, \cite{jewe75} and \cite{spec78}. 
That of discrete hypergroups is such an extension of discrete groups. 
First, we give the definition of discrete hypergroups. 


\begin{defi} \label{defDHG}
Let $H = \set{x_i}_{i \in I}$ be a set of indeterminates indexed by a set 
$I = \set{0, 1, 2, \dots, N}$ for some $N \in \Nonnegative$ or $I = \Nonnegative$. 
The triple $(H, \circ, *)$ consisting of the set $H$, 
a binary operation $\circ: \Complex H \times \Complex H \to \Complex H$ 
called the \textit{convolution} 
and a map $*: \Complex H \to \Complex H$ called the \textit{involution} is said to be 
a \textit{discrete hypergroup} if the following three conditions are satisfied. 

\begin{enumerate}
\item
The triple $(\mathbb{C}H, \circ, *)$ forms a $*$-algebra with the unit $x_0 \in H$. 

\item
The restriction $*|_H$ of the involution $*$ maps $H$ onto $H$. 

\item
For each $i$, $j \in I$, we let $x_i \circ x_j = \sum_{k\in I}q_{i, j}^{k} x_k$ 
with some $q_{i, j}^{k} \in \Complex\ (k\in
I)$. 
Then we have $q_{i, j}^{k} \geq 0$ for all $k\in I$ and 
$\sum_{k\in I} q_{i, j}^{k} = 1$.
Moreover, $q_{i, j}^{0} \neq 0$ if and only if $x_i = x_j^*$. 
\end{enumerate}

We often say $H$, instead of $(H, \circ, *)$, to be a discrete hypergroup as usual. 
If the $*$-algebra $(\Complex H, \circ, *)$ is commutative, 
then the discrete hypergroup $(H, \circ, *)$ is said to be \textit{commutative}. 
If $*|_H$ is the identity map, then the discrete hypergroup $(H, \circ, *)$ is said 
to be \textit{hermitian}. 
\end{defi}


It is easily seen that any hermitian (discrete) hypergroup $(H, \circ, *)$ 
must be commutative. 
See \cite{lass05} for the general theory of discrete hypergroups, for example. 

For convenience, we use the following term for the incomplete discrete hypergroups. 

\begin{defi} 
As the above definition, let $H = \set{x_i}_{i \in I}$ be a set of indeterminates indexed by a set 
$I = \set{0, 1, 2, \dots, N}$ for some $N \in \Nonnegative$ or $I = \Nonnegative$, and $\circ: \Complex H \times \Complex H \to \Complex H$ a binary operation.
The pair $(H, \circ)$ consisting of the set $H$ and the binary operation 
$\circ$ (or $H$) is said to be 
a \textit{pre-hypergroup} if the following two conditions are satisfied. 

\begin{enumerate}
\item The pair $(\Complex H, \circ)$ forms an algebra with the unit $x_0 \in H$ 
which is not necessarily associative. 

\item For each $i$, $j \in I$, we let $x_i \circ x_j = \sum_{k\in I}q_{i, j}^{k} x_k$ 
with some $q_{i, j}^{k} \in \Complex\ (k\in
I)$. 
Then we have $q_{i, j}^{k} \geq 0$ for all $k\in I$ and 
$\sum_{k\in I} q_{i, j}^{k} = 1$. 
Moreover, $q_{i, j}^{0} \neq 0$ if and only if $x_i = x_j$. 
\end{enumerate}
\end{defi} 

Equipped with the involution $*$ given by the conjugate-linear extension 
of the identity map, 
a commutative and associative pre-hypergroup becomes a hermitian hypergroup. 


\subsection{Wildberger's construction}

Wildberger \cite{wild94, wild95} has introduced the idea 
that random walks on a certain kind of graphs can produce hermitian hypergroups. 
Let us recall his idea in the formulated form. 

Given a pointed graph $(\Gamma, v_0)$, 
we take the set $H(\Gamma, v_0) = \set{x_i}_{i \in I(\Gamma, v_0)}$ of indeterminates. 

We define the binary operation $\circ$ on $\mathbb{C} H(\Gamma, v_0)$ by
\[
x_i \circ x_j = \sum_{k \in I(\Gamma, v_0)} p_{i, j}^{k} x_k
\]
for each 
$i$, $j \in I(\Gamma, v_0)$, where
\begin{equation} \label{pijk}
p_{i, j}^{k} = 
\frac{1}{\abs{S_{i}(v_0)}} \sum_{v \in S_{i}(v_0)} 
\frac{\abs{S_{j}(v) \cap S_{k}(v_0)}}{\abs{S_{j}(v)}}
\end{equation}
for each $k \in I(\Gamma, v_0)$. 
Note that the locally finiteness of $\Gamma$ ensures 
that the sum in \eqref{pijk} is a finite sum, 
and the self-centeredness of $\Gamma$ ensures 
that each $p_{i, j}^{k}$ is well-defined. 
(For details, see \cite[Proposition 3.1]{ikka-sawa19}.)
In addition, we define the involution $*$ on $\mathbb{C} H(\Gamma, v_0)$ 
by the unique extension of the identity map on $H(\Gamma, v_0)$. 

\begin{rema}
The sequence $(p_{i, j}^{k})_{k \in I(\Gamma, v_0)}$ forms the distribution of distances 
between the base point $v_0$ and a random vertex $w \in S_{j}(v)$ 
for a random vertex $v \in S_{i}(v_0)$. 
In other words, $p_{i, j}^{k}$ stands for the probability 
that a random walker is in $S_k(v_0)$ after two consecutive jumps 
from $v_0$ to a random vertex $v \in S_i(v_0)$ 
and from $v$ to a random vertex $w \in S_j(v)$. 
\end{rema}

We should note that a pre-hypergroup is not always a hermitian hypergroup. 
In \cite[Theorem 3.3]{ikka-sawa19}, it was shown that a pre-hypergroup $H(\Gamma, v_0)$ becomes 
a hermitian hypergroup whenever $\Gamma$ is a distance-regular graph. 
In addition, there exist some examples of non-distance-regular graphs 
which produce hermitian hypergroups by the above process. 
Here, we refer the reader to \cite{bcn89} for the general theory of distance-regular graphs. 
To check that a pre-hypergroup $H(\Gamma, v_0)$ becomes a hermitian hypergroup, 
by \cite[Proposition 3.2]{ikka-sawa19}, 
it is enough to show that the convolution $\circ$ satisfies the associativity 
and the commutativity. 

For convenience, we use the following terms introduced in \cite{ikka-sawa19}. 

\begin{defi} \label{defHGP}
A pointed graph $(\Gamma, v_0)$ is said to be \textit{hypergroup productive} 
if the pre-hypergroup $H(\Gamma, v_0)$ becomes a hermitian hypergroup.  
If $(\Gamma, v_0)$ is hypergroup productive for every $v_0 \in V$, 
the graph $\Gamma$ is said to be \textit{hypergroup productive}. 
\end{defi}


Actually, we have two conditions for pointed graphs with which the pointed graphs are 
expected to be hypergroup productive. 
Those conditions, named (S1) and (S2) in \cite{endo-mimu-sawa20}, 
are weaker than the distance-regularity. 

\begin{itemize}
\item[(S1)] For every $i \in I(\Gamma, v_0)$ and $v$, $w \in V$, 
we have $\abs{S_i(v)} = \abs{S_i(w)}$. 
(In other words, the function $\abs{S_i(\cdot)}: V \to \Nonnegative$ is a constant 
for every $i \in I(\Gamma, v_0)$.)

\item[(S2)] For every $i$, $j$, $k \in I(\Gamma, v_0)$ and $v$, $w \in S_k(v_0)$, 
we have
\[
\abs{S_i(v) \cap S_j(v_0)} = \abs{S_i(w) \cap S_j(v_0)}.
\] 
(In other words, the function $\abs{S_i(\cdot) \cap S_j(v_0)}: 
S_k(v_0) \to \Nonnegative$ is a constant for every $i$, $j$, $k \in I(\Gamma, v_0)$.)
\end{itemize}

Note that every pointed graph with the condition (S1) must be self-centered. 


In \cite{endo-mimu-sawa20}, it was shown that algebraic computations 
in a pre-hypergroup give some probabilities related to the random walks 
on the original pointed graph under the conditions (S1) and (S2). 


\begin{theo}{\rm(\cite[Theorem 4.5]{endo-mimu-sawa20})} \label{ems}
Suppose a pointed graph $(\Gamma, v_0)$ satisfies the conditions (S1) and (S2). 
Let $H(\Gamma, v_0) = \set{x_i}_{i \in I(\Gamma, v_0)}$ be the pre-hypergroup 
derived from $(\Gamma, v_0)$ and $m \in \Natural$. 

Then, for $i_1$, $i_2$, $\dots$, $i_m \in I(\Gamma, v_0)$, we have 
\begin{align}
& (( \cdots ((x_{i_1} \circ x_{i_2}) \circ x_{i_3}) \circ \cdots ) \circ x_{i_{m-1}}) \circ x_{i_m}
\label{emseq} \\
=& \sum_{v_1 \in S_{i_1}(v_0)} \sum_{v_2 \in S_{i_2}(v_1)} \dots 
\sum_{v_m \in S_{i_m}(v_{m -1})} 
\frac{1}{\prod_{j =1}^{m} \abs{S_{i_j}(v_{j -1})}} x_{d(v_0, v_m)}. \nonumber
\end{align}
\end{theo}

If $H(\Gamma, v_0)$ is an associative algebra in Theorem \ref{ems}, 
one may remove all the parentheses from the left hand side of \eqref{emseq}. 
The coefficient of $x_k$ in the right hand side of \eqref{emseq} is 
the conditional probability that a random walker reaches a vertex in $S_k(v_0)$ 
under $m$-step jumps 
$v_0 \xrightarrow{i_1} v_1 \xrightarrow{i_2} \cdots \xrightarrow{i_m} v_m$. 


\begin{rema}

The distance-regularity requires a graph $\Gamma$ to satisfy the following conditions. 
\begin{itemize}
\item The function 
\begin{equation*}
\abs{S_i(\cdot) \cap S_j(\cdot)}: \setcond{(v, w) \in V \times V}{d(v, w) = k} \to 
\Nonnegative 
\end{equation*} 
\end{itemize}
is a constant for every $i$, $j$, $k \in I(\Gamma, v_0)$.

By \cite[Lemma 2.3]{endo-mimu-sawa20}, any pointed graph consisting 
of a Cayley graph $G$ and the base point $e$ (the unit element of $G$) 
automatically satisfies the condition (S1). 
According to \cite[Lemma 2.4]{endo-mimu-sawa20}, 
we find that the pointed Cayley graph $(G, e)$ becomes distance-regular 
if $(G, e)$ satisfies the condition (S2). 

There can be found in \cite[Example 4.6]{endo-mimu-sawa20} 
or \cite[Subsection 4.2, Fig. 6 with top black base point]{ikka-sawa19} 
several examples of hypergroup productive pointed graphs $(\Gamma, v_0)$ 
not distance-regular but satisfying (S1) and (S2).  
All examples of such pointed graphs that have been published are hypergroup productive, 
and it is an open problem whether any pointed graphs with (S1) and (S2) are 
hypergroup productive. 
\end{rema}


For a pointed graph $(\Gamma, v_0)$ and $k \in I(\Gamma, v_0)$, a transition matrix
\begin{equation*} 
P_k = (p_{k, i}^{j})_{i, j \in I(\Gamma, v_0)} 
\end{equation*} 
is called the \textit{$k$-transition matrix} associated with $(\Gamma, v_0)$, 
where each $p_{k, i}^{j}$ is the probability as in \eqref{pijk}. 

Consider the Hilbert space $\ell^{2}(I(\Gamma, v_0))$ 
of the square summable sequences on $I(\Gamma, v_0)$. 
Then each $P_k$ can be regarded as a bounded operator 
on $\ell^{2}(I(\Gamma, v_0))$ defined by 
\begin{equation*} 
P_k(\xi) = {}^t({}^t\xi P_k) = \left( \sum_{l \in I(\Gamma, v_0)} p_{k, l}^{n} \xi_l \right)_{n \in I(\Gamma, v_0)}, 
\end{equation*} 
where $\xi = (\xi_n)_{n \in I(\Gamma, v_0)} \in \ell^{2}(I(\Gamma, v_0))$. 
An estimate of the operator norm $\norm{P_k}$ was given 
in \cite[Theorem 5.3]{endo-mimu-sawa20}. 
Also in \cite{endo-mimu-sawa20}, it was been shown for all $i$, $j \in I(\Gamma, v_0)$ 
that $P_i$ and $P_j$ commute and satisfy 
\begin{equation*} 
P_i P_j = \sum_{k \in I(\Gamma, v_0)} p_{i, j}^{k} P_k 
\end{equation*} 
when $(\Gamma, v_0)$ is hypergroup productive. 



This fact leads us to the following corollary, 
which tells us that the family $\set{P_k}_{k \in I(\Gamma, v_0)}$ 
of $k$-transition matrices describes the distance distributions of a random walk 
on $\Gamma$ starting from $v_0$. 

\begin{coro}{\rm{(\cite[Corollary 5.8]{endo-mimu-sawa20})}} \label{maincoro}
Assume that a hypergroup productive pointed graph $(\Gamma, v_0)$ 
with the conditions (S1) and (S2) produces the hermitian hypergroup 
$H(\Gamma, v_0)$. 

Then we have
\begin{equation*}
P_{i_m} \cdots P_{i_1} = \sum_{k \in I(\Gamma, v_0)} p_{i_m, \dots, i_1}^{k} P_k
\end{equation*}
for all $i_1$, $\dots$, $i_m \in I(\Gamma, v_0)$, 
where $p_{i_m, \dots, i_1}^{k}$ is the coefficient of $x_k$ in \eqref{emseq}. 
In particular, the sequence $\xi = (\delta_{n, 0})_{n \in I(\Gamma, v_0)} \in 
\ell^{2}(I(\Gamma, v_0))$ can extract $p_{i_m, \dots, k_1}^{k}$ by 
$(P_{i_m} \cdots P_{i_1} \xi)_{k} = p_{i_m, \dots, i_1}^{k}$. 
\end{coro}


\section{Adjacency matrix $A^{(k)}$ and transition matrix $P_k$} \label{secMain}

In this section, we will prove the main theorem stated in Section \ref{secIntro}. 
To prove it, we use the relationship between $A^{(k)}$ and $P_k$, 
clarified in Theorem \ref{PD=DA}. 

To begin with, we introduce the modified form of $A^{(k)}$. 
Let $\mu_k = \abs{S_k(v_0)}$ for a pointed graph $(\Gamma, v_0)$. 
Note that each vertex $v \in V$ satisfies $\mu_k = \abs{S_k(v)}$ 
under the condition (S1). 


\begin{prop}
If a pointed graph $(\Gamma, v_0)$ satisfies the condition (S1), then 
\begin{equation*}
A_k = \frac{1}{\mu_k} A^{(k)}
\end{equation*}
is a doubly stochastic matrix.
\end{prop}

\begin{proof}
Recall that we have defined the $k$-adjacency matrix $A^{(k)}$ 
associated with $\Gamma$ as $A^{(k)} = (\delta_{k, d(x, y)})_{x, y \in V}$. 
Hence, representing the entries of $A_k$ as $A_k = (A_k(x, y))_{x, y \in V}$, we have 
\begin{equation*}
\sum_{x \in V} A_k(x, y) = 
\frac{1}{\mu_k} \sum_{x \in V} \delta_{k, d(x, y)} = 
\frac{1}{\mu_k} \abs{S_k(y)} = 
1 
\end{equation*}
and 
\begin{equation*} 
\sum_{y \in V} A_k(x, y) = 
\frac{1}{\mu_k} \sum_{y \in V} \delta_{k, d(x, y)} = 
\frac{1}{\mu_k} \abs{S_k(x)} = 
1 
\end{equation*} 
for all $x$, $y \in V$. 
This means that $A_k$ is a doubly stochastic matrix. 
\end{proof}


Here we define an important mapping which connects $A_k$ and $P_k$. 

\begin{defi}
Let $\Gamma$ be a graph with the vertex set $V$ and $v_0 \in V$ the base point. 
We define the map $D: \mathbb{C}^{V} \to \mathbb{C}^{I(\Gamma, v_0)}$ by 
\begin{equation*}
D(\bm{s}) = 
\left( \sum_{v \in S_i(v_0)} s_v \right)_{i \in I(\Gamma, v_0)}
\end{equation*}
for each $\bm{s} = (s_v)_{v \in V} \in \mathbb{C}^{V}$.
\end{defi}

Note that if $|V|=\infty$ the set $\mathbb{C}^V$ is the vector space consisting of all sequences with the infinite length $|V|=\infty$ indexed by $V$.
The following fact can be easily seen, but we are going to use it 
to obtain the main theorem. 

\begin{prop}
The map $D: \mathbb{C}^{V} \to \mathbb{C}^{I(\Gamma, v_0)}$ is 
a surjective linear map.
\end{prop}

\begin{proof}
It is obvious that $D$ is a linear map. 

To prove the surjectivity of $D$, let an arbitrary $\bm{t} = (t_i)_{i \in I(\Gamma, v_0)} \in 
\mathbb{C}^{I(\Gamma, v_0)}$ be given. 
Take an arbitrary vertex $v_i$ from $S_i(v_0)$ for each $i \in I(\Gamma, v_0)$ and  
define $\bm{s} = (s_v)_{v \in V} \in \mathbb{C}^{V}$ by 
\begin{equation*}
s_v = \begin{cases}
t_i & (v = v_i \ \text{for some} \ i \in I(\Gamma, v_0)), \\
0 & (v \neq v_i \ \text{for any} \ i \in I(\Gamma, v_0)).
\end{cases}
\end{equation*}
Then we have
\begin{equation*} 
D(\bm{s}) = 
\left( \sum_{v \in S_i(v_0)} s_v \right)_{i \in I(\Gamma, v_0)} = 
(t_i)_{i \in I(\Gamma, v_0)} = 
\bm{t}. 
\end{equation*} 
Hence $D$ is a surjection. 
\end{proof}


Next, we make an observation of $p_{i, j}^{k}$'s, 
the entries of the transition matrices associated with $(\Gamma, v_0)$. 

\begin{lemm} \label{commlemm} 
Suppose that a pointed graph $(\Gamma, v_0)$ satisfies the conditions (S1) and (S2). 
Then, for all $i$, $j$, $k \in I(\Gamma, v_0)$ and $z \in S_k(v_0)$, we have 
\begin{equation} \label{pijklemm}
p_{i, j}^k = \frac{\mu_k \abs{S_{j}(z) \cap S_{i}(v_0)}}{\mu_i \mu_j},
\end{equation}
where $p_{i, j}^k$ is defined in \eqref{pijk}. 
Thus we have $p_{i, j}^{k} = p_{j, i}^{k}$ if and only if 
\begin{equation*} 
\abs{S_i(z) \cap S_j(v_0)} = \abs{S_j(z') \cap S_i(v_0)} 
\end{equation*}
holds for some $z$, $z' \in S_k(v_0)$. 
\end{lemm}

\begin{proof}
The condition (S2) has us find that, for an arbitrary $z \in S_k(v_0)$, 
\begin{align*}
&\sum_{v \in S_i(v_0)} \abs{S_j(v) \cap S_k(v_0)} = \abs{\setcond{(v, w) \in S_i(v_0) \times S_k(v_0)}{d(v, w) = j}} \\
=& \abs{\setcond{(w, v) \in S_k(v_0) \times S_i(v_0)}{d(w, v) = j}} = \sum_{w \in S_k(v_0)} \abs{S_j(w) \cap S_i(v_0)} \\
=& \mu_k \abs{S_j(z) \cap S_i(v_0)}. 
\end{align*} 
This computation and the conditions (S1) and (S2) yield the equations 
\begin{align*} 
&p_{i, j}^{k} = 
\frac{1}{\abs{S_i(v_0)}} \sum_{v \in S_i(v_0)} \frac{\abs{S_j(v) \cap S_k(v_0)}}{\abs{S_j(v)}} = 
\frac{1}{\mu_i \mu_j} \sum_{v \in S_i(v_0)} \abs{S_j(v) \cap S_k(v_0)} \\
=& \frac{\mu_k \abs{S_j(z) \cap S_i(v_0)}}{\mu_i \mu_j}.
\end{align*} 
This completes the proof.
\end{proof}

We now stand at the stage of giving the relationship between $A_k$ and $P_k$. 
The following theorem provides a criterion for the determination 
of the hypergroup productivity. 

\begin{theo} \label{PD=DA}
Suppose that a pointed graph $(\Gamma, v_0)$ satisfies the conditions 
(S1) and (S2). 
Then we have the commutativity of the pre-hypergroup $H(\Gamma, v_0)$ 
if and only if $P_h D = D A_h$ holds for every $h \in I(\Gamma, v_0)$. 
\end{theo}

\begin{proof}
First, we assume that $H(\Gamma, v_0)$ is commutative. 
Then, $p_{i, j}^{k} = p_{j, i}^{k}$ holds for every $i$, $j$, $k \in I(\Gamma, v_0)$. 

Let $\set{\bm{e}_v}_{v \in V}$ and $\set{\bm{f}_i}_{i \in I(\Gamma, v_0)}$ be 
the standard bases of $\mathbb{C}^V$ and $\mathbb{C}^{I(\Gamma, v_0)}$,  respectively. 
What we have to show is that $P_k D \bm{e}_v = D A_k \bm{e}_v$ holds 
for every $v \in V$. 

Fixing $k \in I(\Gamma, v_0)$ and $v \in V$ arbitrarily, 
we can calculate $D A_k \bm{e}_v$ as follows. 
\begin{align}
D A_k \bm{e}_v =& 
\frac{1}{\mu_k} D A^{(k)} \bm{e}_v 
= \frac{1}{\mu_k} D \left( \delta_{k, d(x, v)} \right)_{x \in V}  
= \frac{1}{\mu_k} 
\left( \sum_{x \in S_i(v_0)} \delta_{k, d(x, v)} \right)_{i \in I(\Gamma, v_0)} \label{DAe} \\ 
=& \frac{1}{\mu_k} \left( \abs{S_k(v) \cap S_i(v_0)} \right)_{i \in I(\Gamma, v_0)} 
= \left( \frac{\abs{S_k(v) \cap S_i(v_0)}}{\mu_k} \right)_{i \in I(\Gamma, v_0)}. \nonumber 
\end{align} 

On the other hand, we can calculate $P_k D \bm{e}_v$ as follows. 
\begin{align}
&P_k D \bm{e}_v 
= P_k \left( \sum_{w \in S_i(v_0)} \delta_{v, w} \right)_{i \in I(\Gamma, v_0)} 
= P_k \bm{f}_{d(v_0, v)} \label{PDe} \\
=& \left( p_{k, d(v_0, v)}^{i} \right)_{i \in I(\Gamma, v_0)} 
= \left( \frac{\mu_i \abs{S_{d(v_0, v)}(w_i) \cap S_k(v_0)}}{\mu_k \mu_{d(v_0, v)}} 
\right)_{i \in I(\Gamma,v_0)}, \nonumber 
\end{align}
where each $w_i \in S_i(v_0)$ is arbitrarily taken. 
We have used \eqref{pijklemm} to obtain the fourth equality. 

It remains to show the equation 
\begin{equation} \label{eqRemain}
\frac{\abs{S_k(v) \cap S_i(v_0)}}{\mu_k} = 
\frac{\mu_i \abs{S_{d(v_0, v)}(w_i) \cap S_k(v_0)}}{\mu_k \mu_{d(v_0, v)}} 
\end{equation} 
for every $i \in I(\Gamma, v_0)$. 
As with the proof of Lemma \ref{commlemm}, we have 
\begin{align} 
& \mu_{d(v_0, v)} \abs{S_k(v) \cap S_i(v_0)}  = \sum_{y \in S_{d(v_0, v)}(v_0)} \abs{S_k(y) \cap S_i(v_0)} \label{1} \\ 
=& \abs{\setcond{(y, z) \in S_{d(v_0, v)}(v_0) \times S_i(v_0)}{d(y, z) = k}} \nonumber \\ 
=& \abs{\setcond{(z, y) \in S_i(v_0) \times S_{d(v_0, v)}(v_0)}{d(z, y) = k}} \nonumber \\ 
=& \sum_{z \in S_i(v_0)} \abs{S_k(z) \cap S_{d(v_0, v)}(v_0)} = \mu_i \abs{S_k(w_i) \cap S_{d(v_0, v)}(v_0)} \nonumber \\ 
=& \mu_i \abs{S_{d(v_0, v)}(w_i) \cap S_k(v_0)}. \nonumber  
\end{align}
The last equality follows the assumption $p_{k, d(v_0, v)}^{i} = p_{d(v_0, v), k}^{i}$ and 
Lemma \ref{commlemm}. 
These computations ensure the equation \eqref{eqRemain} 
and we obtain that $P_k D \bm{e}_v = D A_k \bm{e}_v$. 


Conversely, assume that $P_h D = D A_h$ holds for every $h \in I(\Gamma, v_0)$. 
Fixing arbitrary $i$, $j$, $k \in I(\Gamma, v_0)$, 
and we should show the equality $p_{i, j}^{k} = p_{j, i}^{k}$. 



Under the assumption $P_j D = D A_j$, 
the above calculations \eqref{DAe} and \eqref{PDe} imply that 
\begin{align} 
&\frac{\abs{S_j(v) \cap S_i(v_0)}}{\mu_j} 
= (D A_j \bm{e}_v)_{i} 
= (P_j D \bm{e}_v)_{i} 
= p_{j, d(v_0, v)}^{i} \nonumber \\ 
=& \frac{1}{\abs{S_j(v_0)}} \sum_{w \in S_j(v_0)} 
\frac{\abs{S_{d(v_0, v)}(w) \cap S_i(v_0)}}{\abs{S_{d(v_0, v)}(w)}} 
= \frac{\abs{S_{d(v_0, v)}(z) \cap S_i(v_0)}}{\mu_{d(v_0, v)}}\nonumber
\end{align} 
for all $v \in V$ and $z \in S_j(v_0)$. 
If $v$ belongs to $S_k(v_0)$, then we have 
\begin{align*}
& \abs{S_j(v) \cap S_i(v_0)} 
= \frac{\mu_j \abs{S_{d(v_0, v)}(z) \cap S_i(v_0)}}{\mu_{d(v_0, v)}} = \frac{\mu_j \abs{S_k(z) \cap S_i(v_0)}}{\mu_k} \\
=& \frac{\mu_i \abs{S_j(w_i) \cap S_k(v_0)}}{\mu_k} 
= \frac{\mu_k \abs{S_i(v) \cap S_j(v_0)}}{\mu_k} = \abs{S_i(v) \cap S_j(v_0)}, 
\end{align*} 
where $w_i \in S_i(v_0)$ is arbitrarily taken. 
The second equality derives from $d(v_0, v) = k$. 
The third and fourth equalities can be obtained by the calculations similar to \eqref{1}. 
Appealing to the second assertion of Lemma \ref{commlemm}, 
we complete the proof of $p_{i, j}^{k} = p_{j, i}^{k}$. 
\end{proof}


Of course we have the commutativity of $H(\Gamma, v_0)$ 
when $(\Gamma, v_0)$ is hypergroup productive. 
It is known that $(\Gamma, v_0)$ must be hypergroup productive 
if $\Gamma$ is distance-regular, 
so we obtain the following corollary. 

\begin{coro}
If $\Gamma$ is a distance-regular graph, 
for an arbitrary base point $v_0$, 
we have $P_k D = D A_k$ for all $k \in I(\Gamma, v_0)$. 
\end{coro}


\begin{exam}
Let $\Gamma$ be the distance-regular Cayley graph 
associated with the finite group $\mathbb{Z}/4\mathbb{Z} = 
\set{\overline{0}, \overline{1}, \overline{2}, \overline{3}}$ 
and the generating set $\set{\overline{1}, \overline{3}}$. 
Figure 1 describes the vertices and the edges of $\Gamma$. 
For each $k \in I(\Gamma, \overline{0}) = \set{0, 1, 2}$, 
the $k$-adjacency matrix $A^{(k)} = 
(\delta_{k, d(x, y)})_{x, y \in \mathbb{Z}/4\mathbb{Z}}$ associated with $\Gamma$ 
and the $k$-transition matrix $P_k = (p_{k, i}^{j})_{i, j \in I(\Gamma, \overline{0})}$ 
associated with the pointed graph $(\Gamma, \overline{0})$ are computed as 
\begin{align*}
A^{(0)} &= 
\begin{pmatrix} 
1 & 0 & 0 & 0 \\ 
0 & 1 & 0 & 0 \\ 
0 & 0 & 1 & 0 \\ 
0 & 0 & 0 & 1 
\end{pmatrix}, & 
A^{(1)} &= 
\begin{pmatrix} 
0 & 1 & 1 & 0 \\ 
1 & 0 & 0 & 1 \\ 
1 & 0 & 0 & 1 \\ 
0 & 1 & 1 & 0 
\end{pmatrix}, & 
A^{(2)} &= 
\begin{pmatrix} 
0 & 0 & 0 & 1 \\ 
0 & 0 & 1 & 0 \\ 
0 & 1 & 0 & 0 \\ 
1 & 0 & 0 & 0 
\end{pmatrix}, \\
P_0 &= 
\begin{pmatrix} 
1 & 0 & 0 \\ 
0 & 1 & 0 \\ 
0 & 0 & 1 
\end{pmatrix}, & 
P_1 &= 
\begin{pmatrix} 
0 & 1 & 0 \\ 
\frac{1}{2} & 0 & \frac{1}{2} \\ 
0 & 1 & 0 
\end{pmatrix}, & 
P_2 &= 
\begin{pmatrix} 
0 & 0 & 1 \\ 
0 & 1 & 0 \\ 
1 & 0 & 0 
\end{pmatrix}.
\end{align*}
The matrices $A^{(k)}$'s and $P_k$'s mutually commute, respectively. 
We can straightforwardly check that $P_k D = D A_k$ for $k = 0$, $1$, $2$. 
\[
\begin{xy} 
(0,0)*{\circ}="A",
(-15,0)*{\circ}="B",
(0,-15)*{\circ}="C",
(-15,-15)*{\circ}="D",
{ "A" \ar @{-} "B" },
{ "A" \ar @{-} "C" },
{ "C" \ar @{-} "D" },
{ "B" \ar @{-} "D" },
(-18,3)*{\overline{0}},
(3,3)*{\overline{1}},
(3,-18)*{\overline{2}},
(-18,-18)*{\overline{3}},
(-7.5,-25)*{\mbox{{\rm Figure }}1},
\end{xy}
\]
\end{exam}



We will give a necessary and sufficient condition 
for the hypergroup productivity in the view point of adjacency matrices. 
That condition will enable us to easily show that a pointed graph $(\Gamma, v_0)$ 
with (S1) and (S2) is hypergroup productive when the diameter of $\Gamma$ is two. 
We will also determine the structures of the hermitian hypergroups derived 
from such pointed graphs by a combinatoric method. 

Since the fact in the following lemma can be found 
in \cite[Remark 5.1]{endo-mimu-sawa20} without any proofs, 
we shall give a proof for convenience.


\begin{lemm} \label{assocomm}
Let $H = \set{x_i}_{i \in I}$ be a commutative pre-hypergroup and $Q_k = (q_{k, i}^{j})_{i,j\in
I}$ for each $k\in
I$, 
where each $q_{i, j}^{k}$ is given by $x_i \circ x_j = \sum_{k \in I} q_{i, j}^{k} x_k$. 
Then $Q_i Q_j = Q_j Q_i$ holds for every $i$, $j \in I$ 
if and only if $H$ satisfies the associativity and forms a hermitian hypergroup.
\end{lemm} 

\begin{proof}
We find that the entries of $Q_i Q_j$ and $Q_j Q_i$ are given by 
\begin{align} 
&(Q_i Q_j)_{l, m} 
= \sum_{h \in I} q_{i, l}^{h} q_{j, h}^{m} 
= \sum_{h \in I} q_{i, l}^{h} q_{h, j}^{m} \label{eqQiQj} \\ 
\intertext{and}
&(Q_j Q_i)_{l, m} 
= \sum_{h \in I} q_{j, l}^{h} q_{i, h}^{m} 
= \sum_{h \in I} q_{l, j}^{h} q_{i, h}^{m}, \label{eqQjQi}
\end{align} 
respectively. 
Note that we have used the commutativity of $H$ to get the second equality 
in each calculations. 

On the other hand, $x_i \circ (x_l \circ x_j) = (x_i \circ x_l) \circ x_j$ 
for $i$, $j$, $l \in I$ is equivalent to the equations 
\begin{equation} \label{asso} 
\sum_{h \in I} q_{l, j}^{h} q_{i, h}^{m} = 
\sum_{h \in I} q_{i, l}^{h} q_{h, j}^{m} 
\end{equation} 
for all $m \in I$. 

Therefore, we can deduce that the associativity of $H$ is equivalent to 
$Q_i Q_j = Q_j Q_i$ for all $i$, $j \in I$ 
from \eqref{eqQiQj}, \eqref{eqQjQi} and \eqref{asso}. 
\end{proof}

When $H = H(\Gamma, v_0)$ derives from a pointed graph $(\Gamma, v_0)$, 
the matrix $Q_k$ in Lemma \ref{assocomm} is replaced by the $k$-transition matrix 
$P_k = (p_{k, i}^{j})_{i, j \in I(\Gamma, v_0)}$ associated with $(\Gamma, v_0)$. 


\begin{theo}\label{assoA_k}
Suppose that a pointed graph $(\Gamma, v_0)$ satisfies the conditions (S1) and (S2). 
Then, $(\Gamma, v_0)$ is hypergroup productive 
if and only if $D A_k A_l = D A_l A_k$ holds for every $k$, $l \in I(\Gamma, v_0)$. 
\end{theo}

\begin{proof}
Suppose that $D A_k A_l = D A_l A_k$ holds for every $k$, $l \in I(\Gamma, v_0)$. 

Let $\set{\bm{e}_v}_{v \in V}$ be the standard basis of $\mathbb{C}^V$. 
Then we have 
\begin{align*} 
& D A_k A_l \bm{e}_{v_0} 
= \frac{1}{\mu_k \mu_l} D A^{(k)} A^{(l)} \bm{e}_{v_0} 
= \frac{1}{\mu_k \mu_l} D A^{(k)} (\delta_{l, d(v, v_0)})_{v \in V} \\ 
=& \frac{1}{\mu_k \mu_l} D \left( \sum_{w \in V} \delta_{k, d(v, w)} \delta_{l, d(w, v_0)} 
\right)_{v \in V} 
= \frac{1}{\mu_k \mu_l} D (\abs{S_k(v) \cap S_l(v_0)})_{v \in V} \\ 
=& \frac{1}{\mu_k \mu_l} \left( \sum_{v \in S_i(v_0)} \abs{S_k(v) \cap S_l(v_0)} 
\right)_{i \in I(\Gamma, v_0)} \\ 
= & \frac{1}{\mu_k \mu_l} \left( \mu_i \abs{S_k(v_i) \cap S_l(v_0)} 
\right)_{i \in I(\Gamma, v_0)}, 
\end{align*}
where each $v_i \in S_i(v_0)$ is arbitrarily taken. 
Similar calculations yield 
\begin{equation*} 
D A_l A_k \bm{e}_{v_0} = 
\frac{1}{\mu_k \mu_l} \left( \mu_i \abs{S_l(v_i) \cap S_k(v_0)} 
\right)_{i \in I(\Gamma, v_0)}. 
\end{equation*} 

By the assumption $D A_k A_l = D A_l A_k$, we have
\begin{equation*} 
\abs{S_k(v_i) \cap S_l(v_0)} = \abs{S_l(v_i) \cap S_k(v_0)} 
\end{equation*} 
for all $i \in I(\Gamma, v_0)$. 
Hence Lemma \ref{commlemm} implies that $p_{k, l}^{i} = p_{l, k}^{i}$ 
and the commutativity of $H(\Gamma, v_0)$. 

Using Theorem \ref{PD=DA} repeatedly, we have 
\begin{equation} \label{pa}
P_k P_l D = P_k D A_l = D A_k A_l = D A_l A_k = P_l D A_k = P_l P_k D. 
\end{equation}
Since $D$ is surjective, the matrices $P_k$ and $P_l$ commute. 
We can apply Lemma \ref{assocomm} to obtain the associativity 
of the pre-hypergroup $H(\Gamma, v_0)$, 
and find that $H(\Gamma, v_0)$ forms a hermitian hypergroup. 

Conversely, if the pointed graph $(\Gamma, v_0)$ is hypergroup productive, 
then Lemma \ref{assocomm} and the equations \eqref{pa} imply 
that $D A_k A_l = D A_l A_k$ holds for every $k$, $l \in I(\Gamma, v_0)$. 
\end{proof}


Theorem \ref{assoA_k} leads us to the theorem introduced in Section \ref{secIntro}. 

\begin{coro} \label{corMain}
Suppose that a pointed graph $(\Gamma, v_0)$ satisfies the conditions (S1) and (S2). 
If the $k$-adjacency matrices $A^{(k)}$'s associated with $\Gamma$ 
mutually commute, 
then $(\Gamma, v_0)$ is hypergroup productive. 
\end{coro} 

\begin{proof} 
We have from the assumption 
\begin{equation*} 
D A_k A_l = 
\frac{1}{\mu_k \mu_l} D A^{(k)} A^{(l)} = 
\frac{1}{\mu_k \mu_l} D A^{(l)} A^{(k)} = 
D A_l A_k 
\end{equation*} 
for all $k$, $l \in I(\Gamma, v_0)$. 
Applying Theorem \ref{assoA_k} makes the end of the proof. 
\end{proof}


\section{The case of diameter two} \label{secApplication}

In this section, we discuss mainly the case where the diameter of $\Gamma$ is two 
under the conditions (S1) and (S2). 
We can apply Corollary \ref{corMain} to such a pointed graph as follows. 

\begin{coro} \label{diam2hypg} 
Any pointed graph $(\Gamma, v_0)$ with the conditions (S1) and (S2) 
such that $I(\Gamma, v_0) = \set{0, 1, 2}$ is hypergroup productive. 
\end{coro}

\begin{proof}
To apply Corollary \ref{corMain}, we should check that $A^{(k)}$ and $A^{(l)}$ commute 
for all $k$ and $l$ such that $0 \leq k \leq l \leq 2$. 

Recall that $A^{(0)} = \bm{1}$, and we can easily check that 
$A^{(k)} A^{(l)} = A^{(l)} A^{(k)}$ holds except for the case of $(k, l) = (1, 2)$. 
Thus what we have to show is $A^{(1)} A^{(2)} = A^{(2)} A^{(1)}$. 

Let $J$ denote the square matrix of order $\abs{V}$ all of whose entries are $1$. 
Then the condition (S1) makes $A^{(1)}$ and $J$ commute. 
Therefore, $A^{(2)} = J - A^{(0)} - A^{(1)}$ and $A^{(1)}$ must commute. 
Apply Corollary \ref{corMain} to complete the proof. 
\end{proof}


We can determine the structure of the hermitian hypergroup derived 
from a pointed graph as in Corollary \ref{diam2hypg}. 
To determine the structure of a hermitian hypergroup $H$ is 
to compute all the convolutions in the form of $x_i \circ x_j$. 

\begin{theo} \label{diam2}
Let $(\Gamma, v_0)$ be a pointed graph with the conditions (S1) and (S2) 
such that $I(\Gamma, v_0) = \set{0, 1, 2}$. 
Then the structure of the hermitian hypergroup 
$H(\Gamma, v_0) = \set{x_0, x_1, x_2}$ is given by 
\begin{align*} 
&x_1 \circ x_1 
= \frac{1}{\mu_1} x_0 + \frac{m}{\mu_1} x_1 + \frac{\mu_1-1-m}{\mu_1} x_2, \\ 
&x_1 \circ x_2 = x_2 \circ x_1 
= \frac{\mu_1-1-m}{\mu_2} x_1 + \frac{\mu_2-\mu_1+1+m}{\mu_2} x_2, \\ 
&x_2 \circ x_2 
= \frac{1}{\mu_2} x_0 + \left( \frac{\mu_1}{\mu_2} - \frac{\mu_1(\mu_1-1-m)}{\mu_2^2} 
\right) x_1 \\ 
&\hspace{160pt} + \left( 1 - \frac{1}{\mu_2} - \frac{\mu_1}{\mu_2} + 
\frac{\mu_1(\mu_1-1-m)}{\mu_2^2} \right) x_2, 
\end{align*}
where $m = \abs{S_1(v_1) \cap S_1(v_0)}$ with $v_1 \in S_1(v_0)$ arbitrarily taken. 
\end{theo}

\begin{proof}
By the definition \eqref{pijk} of $p_{i, j}^k$ and the conditions $(S1)$ and $(S2)$, $p_{i, j}^{k}$ can be written as 
\begin{equation*} 
p_{i, j}^{k} = \frac{\abs{S_j(v_i) \cap S_k(v_0)}}{\mu_j} 
\end{equation*} 
for $i$, $j$, $k \in I(\Gamma, v_0)$, where $v_i \in S_i(v_0)$ is arbitrarily taken. 

Put $m_{i, j}^{k} = \mu_j p_{i, j}^{k} = \abs{S_j(v_i) \cap S_k(v_0)}$. 
Then we have 
\begin{align}
& m_{i, j}^{0} + m_{i, j}^{1} + m_{i, j}^{2} = \mu_j \label{qij} \\
\intertext{and} 
& m_{i, 0}^{k} + m_{i, 1}^{k} + m_{i, 2}^{k} = \mu_k \label{qik} 
\end{align} 
for all $i$, $j$, $k \in \set{0, 1, 2}$. 
The former relation \eqref{qij} provides $m_{1, 1}^{2} = \mu_1 - 1 - m_{1, 1}^{1}$ 
since $m_{1, 1}^{0} = 1$, 
and the latter relation \eqref{qik} provides
\[m_{1, 2}^{1} = \mu_1 - 1 - m_{1, 1}^{1}
\]
since $m_{1, 0}^{1} = 1$. 
Furthermore, we have $m_{1, 0}^{2} = m_{1, 2}^{0} = m_{2, 1}^{0} = 0$ and 
\begin{equation*} 
m_{1, 2}^{2} = \mu_2 - m_{1, 0}^{2} - m_{1, 1}^{2} = \mu_2 - \mu_1 + 1 + m_{1, 1}^{1}. 
\end{equation*} 

In addition, we can find that $\mu_1 m_{1, 1}^{2} = \mu_2 m_{2, 1}^{1}$ holds. 
Indeed, $m_{1, 1}^{2} = \abs{S_1(v_1) \cap S_2(v_0)}$ represents 
the number of vertices belonging to $S_2(v_0)$ which are adjacent to $v_1$. 
This is equal to the number of edges which start from $v_1$ and reach $S_2(v_0)$. 
Thus the number $\mu_1 m_{1, 1}^{2}$ counts how many edges start from $S_1(v_0)$ 
and reach $S_2(v_0)$ under the condition (S1). 
By the same argument, we find that the number $\mu_2 m_{2, 1}^{1}$ counts 
how many edges start from $S_2(v_0)$ and reach $S_1(v_0)$. 
These two numbers coincide so that we have $\mu_1 m_{1, 1}^{2} = \mu_2 m_{2, 1}^{1}$. 

Thanks to this relation, we can get 
\begin{align*} 
& m_{2, 1}^{1} 
= \frac{\mu_1 m_{1, 1}^{2}}{\mu_2} 
= \frac{\mu_1 (\mu_1 - 1 - m_{1, 1}^{1})}{\mu_2}, \\ 
& m_{2, 1}^{2} 
= \mu_1 - m_{2, 1}^{0} - m_{2, 1}^{1} 
= \mu_1 - \frac{\mu_1 (\mu_1 - 1 - m_{1, 1}^{1})}{\mu_2}. 
\end{align*} 
Noting that $m_{2, 0}^{1} = 0$ and $m_{2, 0}^{2} = 1$, we can also get 
\begin{align*}
& m_{2, 2}^{1} 
= \mu_1 - m_{2, 0}^{1} - m_{2, 1}^{1} 
= \mu_1 - \frac{\mu_1 (\mu_1 - 1 - m_{1, 1}^{1})}{\mu_2}, \\ 
& m_{2, 2}^{2} 
= \mu_2 - m_{2, 0}^{2} - m_{2, 1}^{2} 
= \mu_2 - 1 - \mu_1 + \frac{\mu_1 (\mu_1 - 1 - m_{1, 1}^{1})}{\mu_2}. 
\end{align*} 
We also have $m_{2, 2}^{0} = 1$.

Therefore, putting $m = m_{1, 1}^{1}$, 
we can list the numbers $p_{i, j}^{k} = m_{i, j}^{k}/\mu_j$ as follows: 
\begin{align*}
& p_{1, 1}^{0} = \frac{1}{\mu_1}, \quad 
p_{1, 1}^{1} = \frac{m}{\mu_1}, \quad 
p_{1, 1}^{2} = \frac{\mu_1 - 1 - m}{\mu_1}; \\ 
& p_{1, 2}^{0} = p_{2, 1}^{0} = 0, \quad 
p_{1, 2}^{1} = p_{2, 1}^{1} = \frac{\mu_1 - 1 - m}{\mu_2}; \\ 
& p_{1, 2}^{2} = p_{2, 1}^{2} = \frac{\mu_2 - \mu_1 + 1 + m}{\mu_2}; \\ 
& p_{2, 2}^{0} = \frac{1}{\mu_2}, \quad 
p_{2, 2}^{1} = \frac{\mu_1}{\mu_2} - \frac{\mu_1 (\mu_1 - 1 - m)}{\mu_2^2}, \\ 
& p_{2, 2}^{2} = 1 - \frac{1}{\mu_2} - \frac{\mu_1}{\mu_2} + 
\frac{\mu_1 (\mu_1 - 1 - m)}{\mu_2^2}. 
\end{align*}
(If either $i = 0$ or $j = 0$, we can easily compute $x_i \circ x_j$ 
since $x_0$ is the unit.) 

This completes the proof. 
\end{proof}


\begin{rema}
According to \cite[Section 4]{wild02}, the pre-hypergroup 
$H = \set{c_0, c_1, c_2}$ with the structure 
\begin{align*}
& c_1 \circ c_1 
= \frac{1}{\omega_1} c_0 + \alpha_1 c_1 + \beta_1 c_2, \quad 
c_2 \circ c_2 
= \frac{1}{\omega_2} c_0 + \beta_2 c_1 + \alpha_2 c_2, \\ 
& c_1 \circ c_2 = c_2 \circ c_1 
= \gamma_1 c_1 + \gamma_2 c_2 
\end{align*}
becomes associative if and only if both of the relations 
\begin{equation} \label{wild}
\beta_1 \omega_1 = \gamma_1 \omega_2, \quad 
\beta_2 \omega_2 = \gamma_2 \omega_1 
\end{equation} 
hold. 

This fact enable us to give another proof of Corollary \ref{diam2hypg}. 
Let $(\Gamma, v_0)$ be a pointed graph with (S1) and (S2) 
such that $I(\Gamma, v_0) = \set{0, 1, 2}$, as in Corollary \ref{diam2hypg}. 
Then, by putting 
\begin{align*} 
& \omega_1 = \mu_1, \quad \omega_2 = \mu_2; \quad 
\alpha_1 = \frac{m}{\mu_1}, \quad 
\alpha_2 = 1 - \frac{1}{\mu_2} - \frac{\mu_1}{\mu_2} 
+ \frac{\mu_1(\mu_1-1-m)}{\mu_2^2}; \\ 
& \beta_1 = \frac{\mu_1-1-m}{\mu_1}, \quad 
\beta_2 = \frac{\mu_1}{\mu_2} - \frac{\mu_1 (\mu_1-1-m)}{\mu_2^2}; \\ 
& \gamma_1 = \frac{\mu_1-1-m}{\mu_2}, \quad 
\gamma_2 = \frac{\mu_2-\mu_1+1+m}{\mu_2}, 
\end{align*} 
where $m$ is the number defined in Theorem \ref{diam2}, 
we can easily check that these coefficients satisfy \eqref{wild}. 
Therefore, $H(\Gamma, v_0)$ is associative and forms a hermitian hypergroup. 

Note that $\alpha_i$ can be described by $\omega_i$ and $\beta_i$ for $i = 1$, $2$, 
and hence the relations \eqref{wild} imply 
that any hermitian hypergroup with order three is determined by three parameters $\omega_1,\omega_2$ and $\beta_1$. 
On the other hand, the structure of the hermitian hypergroup $H(\Gamma, v_0)$ 
derived from $(\Gamma, v_0)$ is determined by three parameters 
$\mu_1$, $\mu_2$ and $m$. 
\end{rema}

Finally, we present an example of hypergroup productive pointed graphs 
$(\Gamma,v_0)$ with the conditions (S1) and (S2) such that $\abs{I(\Gamma,v_0)} > 2$ 
and $\Gamma$ is not distance-regular. 

\begin{exam}

Let $(\Gamma,v_0)$ be the pointed graph with $I(\Gamma,v_0)=\{0,1,2,3\}$ drawn as Figure 2. 
We can check that $(\Gamma,v_0)$ produces the hermitian hypergroup with the following structure. 
\begin{align*}
&x_1\circ
x_1=x_2\circ
x_2=\frac{1}{6}x_0+\frac{1}{3}x_1+\frac{1}{2}x_2;\\
&x_1\circ
x_2=x_2\circ
x_1=\frac{1}{2}x_1+\frac{1}{3}x_2+\frac{1}{6}x_3;\\
&x_1\circ
x_3=x_3\circ
x_1=x_2;\quad
x_2\circ
x_3=x_3\circ
x_2=x_1;\quad
x_3\circ
x_3=x_0.
\end{align*}

Also, $(\Gamma, v_0)$ satisfies the conditions (S1) and (S2), and the graph $\Gamma$ is not distance-regular. 
\[
\begin{xy}
(0,7)*{\circ}="A",
(-10,-10)*{\circ}="B",
(10,-10)*{\circ}="C",
(-50,-15)*{\circ}="D",
(50,-15)*{\circ}="E",
(-30,-20)*{\circ}="F",
(30,-20)*{\circ}="G",
(0,-57)*{\circ}="A'",
(-10,-30)*{\circ}="B'",
(10,-30)*{\circ}="C'",
(-50,-35)*{\circ}="D'",
(50,-35)*{\circ}="E'",
(-30,-40)*{\circ}="F'",
(30,-40)*{\circ}="G'",
{ "B" \ar @{-} "B'" },
{ "C" \ar @{-} "C'" },
{ "D" \ar @{-} "D'" },
{ "E" \ar @{-} "E'" },
{ "B" \ar@[red] @{-} "C'" },
{ "C" \ar@[red] @{-} "B'" },
{ "D" \ar@[red] @{-} "B'" },
{ "B" \ar@[red] @{-} "D'" },
{ "C" \ar@[red] @{-} "E'" },
{ "E" \ar@[red] @{-} "C'" },
{ "A" \ar @{-} "B" },
{ "A" \ar @{-} "C" },
{ "A" \ar @{-} "D" },
{ "A" \ar @{-} "E" },
{ "A" \ar @{-} "F" },
{ "A" \ar @{-} "G" },
{ "B" \ar @{-} "C" },
{ "B" \ar @{-} "D" },
{ "C" \ar @{-} "E" },
{ "E" \ar @{-} "G" },
{ "D" \ar @{-} "F" },
{ "F" \ar @{-} "G" },
{ "A'" \ar @{-} "B'" },
{ "A'" \ar @{-} "C'" },
{ "A'" \ar @{-} "D'" },
{ "A'" \ar @{-} "E'" },
{ "A'" \ar @{-} "F'" },
{ "A'" \ar @{-} "G'" },
{ "B'" \ar @{-} "C'" },
{ "B'" \ar @{-} "D'" },
{ "C'" \ar @{-} "E'" },
{ "E'" \ar @{-} "G'" },
{ "D'" \ar @{-} "F'" },
{ "F'" \ar @{-} "G'" },
{ "F" \ar @{-} "F'" },
{ "G" \ar @{-} "G'" },
{ "D" \ar@[green] @{-} "F'" },
{ "F" \ar@[green] @{-} "D'" },
{ "F" \ar@[green] @{-} "G'" },
{ "G" \ar@[green] @{-} "F'" },
{ "G" \ar@[green] @{-} "F'" },
{ "G" \ar@[green] @{-} "E'" },
{ "E" \ar@[green] @{-} "G'" },
(0,11)*{v_0},
(0,-63)*{\mbox{{\rm
Figure }}2},
\end{xy}
\]
\end{exam}

\subsection*{Acknowledgements}
The authors would like to thank Kenta Endo and Ippei Mimura for their helpful comments. This work was supported by JSPS KAKENHI Grant-in-Aid for Research Activity Start-up (No. 19K23403).

\end{document}